\documentclass[11pt,a4paper]{article}
\usepackage{amsmath}
\usepackage{amssymb}  
\usepackage{epsfig}
\usepackage{amsthm}   
\usepackage[mathcal]{eucal}
\usepackage{url}

\usepackage{mathrsfs}

\newcommand{\R}{\mathbb{R}}

\newcommand{\inner}[2]{\langle{#1},{#2}\rangle}
\newcommand{\Inner}[2]{\left\langle{#1},{#2}\right\rangle}
\newcommand{\norm}[1]{\|#1\|}

\newcommand{\tos}{\rightrightarrows}

\newtheorem{theorem}{Theorem}[section]
\newtheorem{lemma}[theorem]{Lemma}
\newtheorem{corollary}[theorem]{Corollary}
\newtheorem{remark}[theorem]{Remark}
\newtheorem{proposition}[theorem]{Proposition}

\newtheorem{definition}[theorem]{Definition}

\begin{document}
\title{A class of 
Fej\'er convergent algorithms,
approximate resolvents and the
 Hybrid Proximal-Extragradient method}

\author{B. F. Svaiter \thanks{IMPA, Estrada Dona Castorina 110,
     22460-320 Rio de Janeiro, Brazil ({\tt benar@impa.br}) \ Partially
     supported by
     CNPq grants 
     302962/2011-5, 
     474944/2010-7, 
     FAPERJ grant
     E-26/102.940/2011 
     and by
     PRONEX-Optimization.}}
\date{}
\maketitle

\begin{abstract}
  A new framework for analyzing Fej\'er convergent algorithms is
  presented. Using this framework we define a very general class
  of Fej\'er convergent algorithms and establish its convergence
  properties.
  We also introduce a new definition of approximations of resolvents
  which preserve some useful features of the exact resolvent, and
  use this concept to present  an unifying view of the
  Forward-Backward splitting method, Tseng's Modified Forward-Backward
  splitting method and Korpelevich's method.
  We show that methods based on families of approximate resolvents
  fall within the aforementioned class of Fej\'er convergent methods.
  We prove that such approximate resolvents are the iteration maps
  of the Hybrid Proximal-Extragradient method.

  \bigskip

  
  \bigskip
  
\end{abstract}

\pagestyle{plain}

\section{Introduction}

In this work we introduce a new framework for analysing Fej\'er convergent
algorithms in Hilbert spaces, by means of recursive inclusions and
sequences of point-to-set maps.
This framework defines a new class of Fej\'er convergent methods,
which is general
enough to encompass, for example, the classical Forward-Backward
splitting method, Tseng's Modified Forward-Backward splitting method
and Korpelevich's method.
Using this framework, we prove for good that convergence with
summable errors is a generic property of a large class of Fej\'er
convergent algorithms.
Therefore, we regard this convergence result (with summable errors) as
a rather negative result, in the sense that it is too generic to
convey useful information on Fej\'er convergent methods.
For sure, this kind of convergence for particular Fej\'er
convergent algorithms lacks particular value.

Another original contribution of this work is the concept of
approximate resolvents of maximal monotone operators.
Approximate resolvents retain the relevant features of exact
resolvents as iteration maps for finding zeros of maximal monotone
operators, their elements are more easily computable and are indeed
calculated in, for instance, the classical Forward-Backward splitting
method, Tseng's Modified Forward-Backward splitting method and
Korpelevich's method.
We prove that any algorithm based on approximate resolvents fall within
the above mentioned class of Fej\'er convergent methods, providing
an unifying framework for establishing their convergence properties.

We present a new transportation formula for cocoercive operators and
use it for establishing that the Forward-Backward splitting method
is a particular case of the Hybrid Proximal-Extragradient method.
The relationship between the Hybrid Proximal-Extragradient method and
approximate resolvents is also discussed.

This work is organized as follows.
In Section~\ref{sec:bas},
we introduce some basic definitions and results.
In Section~\ref{sec:class.fc}, we define a very general class of
Fej\'er convergent by means of recursive inclusions and sequences of
point-to-set maps satisfying two basic properties.
In Section~\ref{sec:apres},
we define approximate resolvents, show that they are the iteration
maps of the Hybrid Proximal-Extragradient method and prove that
methods based on approximate resolvents fall within the aforementioned class
of Fej\'er convergent methods.
In Section~\ref{sec:fv-bw},
we show that the Forward-Backward method is
based on approximate resolvents, which is to say that it is a particular case
of the Hybrid Proximal-Extragradient method.
In Section~\ref{sec:ts.mfb},
we recall that Tseng's Modified Forward-Backward method is a
particular case of the Hybrid Proximal-Extragradient method, which is
to say that it is based on approximate resolvents.
In Section~\ref{sec:kp},
we recall that Korpelevich's method is a particular case of the Hybrid
Proximal-Extragradient method, which is to say that it is based on
approximate resolvents.
In Section~\ref{sec:conc},
we make  some comments.

\section{Basic definitions and results}
\label{sec:bas}

In the first part of this section, we review the concept and
properties of Quasi-Fej\'er convergence, which will be used in our
analysis of a class of Fej\'er convergent methods.
In the second part, we establish the notation concerning point-to-set
maps, which will be used for defining the aforementioned class.

The last part of this section contains the material 
needed to define approximate resolvents and the Hybrid
Proximal-Extragradient (HPE) method, to prove that methods based on
approximate resolvents/HPE method belong to the aforementioned class,
and that some well known decomposition methods are based on
approximate resolvents/HPE method.

As far as we know, this section contains just one original result, namely,
Lemma~\ref{lm:tr.cc}, which is a transportation formula for cocoercive
operators.

\subsection*{Quasi-Fej\'er convergence}

The concept of Quasi-Fej\'er convergence was introduce by
Ermol$'$ev~\cite{MR0302181}  in the context of sequences of random
variables (see also~\cite{MR0246381} and its
translation~\cite{MR0322915}).  We will use a deterministic version of this notion,
considered first in metric spaces by Isuem, Svaiter and
Teboulle~\cite[Definition 4.1]{ITS-Maryland,MR1304625}, in Euclidean
spaces~\cite[Definition 1]{MR1314370}, in Hilbert spaces
in~\cite{MR1617701}, in reflexive Banach spaces in~\cite{MR1613463}.

All this material is now standard knowledge and is included for the
sake of completeness. We do not claim to give \emph{here} any original
contribution to this over-studied concept.
We will use an  arbitrary exponent $p$ just to unify the results related
in the references;  its particular value seems of little importance
as indicated by the next results.

\begin{definition}
  \label{df:qf}
  Let $X$ be a metric space and $0<p<\infty$. A sequence $(x_n)$ in
  $X$ is $p$-Quasi-Fej\'er convergent to $\Omega\subset X$ if, for each
  $x^*\in \Omega$, there exists a non-negative, summable sequence
  $(\rho_n)$ such that
  \[
  d(x^*,x_{n})^p\leq d(x^*,x_{n-1})^p+\rho_n\qquad n=1,2,\dots
  \]
\end{definition}

Note that if $\rho_1=\rho_2=\cdots=0$ in the above definition, we
retrieve the classical definition of Fej\'er convergence and the exponent
$p$ becomes immaterial.
Ermol$'$lev considered the stochastic case with
$p=2$ and the deterministic case was considered in~\cite{ITS-Maryland,MR1304625}
with $p=1$ and in \cite{MR1314370,MR1617701} with $p=2$ in Euclidean
and Hilbert spaces respectively.
The next proposition summarizes the main properties of
Quasi-Fej\'er convergent sequences in metric spaces.

\begin{proposition}
  \label{pr:qf1}
  Let $X$ be a metric space, $p\in (0,\infty)$ and $(x_n)$ be a
  sequence in $X$ which is $p$-Quasi-Fej\'er convergent to
  $\Omega\subset X$ then,
  \begin{enumerate}
  \item if $\Omega$ is non-empty, then $(x_n)$ is bounded;
  \item for any $x^*\in \Omega$ there exists \(
    \lim_{n\to\infty}d(x^*,x_n)<\infty \);
  \item if the sequence $(x_n)$ has a cluster point $x^*\in \Omega$,
    then it converges to such a point.
 \end{enumerate}
\end{proposition}
\begin{proof}
  Take $x^*\in \Omega$ and let $(\rho_n)$ be as in
  Definition~\ref{df:qf}. Then for $n<m$
  \[
  d(x^*,x_m)^p\leq d(x^*,x_n)^p+\sum_{i=n+1}^m\rho_i.
  \]
  Hence
  \[
  \lim\sup_{m\to\infty} d(x^*,x_m)^p\leq d(x^*,x_n)^p+\sum_{i=n+1}^\infty
  \rho_i<\infty,
  \]
  which proves item 1.
  To prove item 2, note that
  $(\rho_n)$ is summable and take the $\lim\inf_{n\to\infty}$
  at the right hand
  side of the  first inequality in the above equation.
  Item 3 follows trivially from item 2.
\end{proof}

Now we recall Opial's Lemma~\cite{MR0211301}, which is useful
for analyzing Quasi-Fej\'er convergence in Hilbert spaces:

\begin{lemma}[Opial]
  \label{lm:op}
  If, in a Hilbert space $X$, the sequence $(x_n)$ is weakly convergent
  to $x^*$, then for any $x\neq x^*$
  \[
  \lim\inf_{n\to\infty}\norm{x_n-x}>  \lim\inf_{n\to\infty}\norm{x_n-x^*}.
  \]
\end{lemma}

The next result was proved in~\cite{MR0410483}, for the case of a
specific sequence generated by an inexact proximal point method with
$p=1$, but the proof presented there is quite general, and we provide
it here for the sake of completeness. 
The idea of using Opial's Lemma seems to be due to H. Brezis.
Latter on this result was explicitly proved for Quasi-Fej\'er
convergent sequences in Hilbert and Banach spaces with $p=2$ in
\cite[Proposition 1]{MR1617701}, \cite[Lemma 2.8]{MR1613463}
respectively.

\begin{proposition}
  \label{pr:qf2}
  If, in a Hilbert space $X$, the sequence $(x_n)$ is $p$-Quasi-Fej\'er
  convergent to $\Omega\subset X$, then it has at most one weak
  cluster point in $\Omega$.
\end{proposition}
\begin{proof}
  If $x^*\in \Omega$ is a weak cluster point of $(x_n)$, then there
  exists a subsequence $(x_{n_k})$ weakly convergent to
  $x^*$. Therefore, using item 2 of Proposition~\ref{pr:qf1} and
  Opial's Lemma, we conclude that for any $x'\in \Omega$, $x'\neq x^*$
\begin{align*}
  \lim_{n\to\infty}\norm{x_n-x'}=
 \lim\inf_{k\to\infty}\norm{x_{n_k}-x'}>  \lim\inf_{k\to\infty}\norm{x_n-x^*}
  = \lim_{n\to\infty}\norm{x_n-x^*}
\end{align*}
 which trivially implies the desired result.
\end{proof}

It is trivial that the specific value of $p\in (0,\infty)$ is
immaterial in the above proofs.  It would be preposterous to claim
that for each $p$ one has a ``specific kind'' of Quasi-Fej\'er
convergence.  We hope to reinforce this point of view with the next
remark.

\begin{remark}
  \label{rm:tr}
  Let $X$ be a metric space, $p\in (0,\infty)$ and $(x_n)$ be a
  sequence in $X$ which is $p$-Quasi-Fej\'er convergent to
  $\Omega\subset X$ then either,
  \begin{enumerate}
  \item $d(x_n,x^*)\to 0$ for some $x^*\in \Omega$;
  \item $(x_n)$ is $q$-Quasi-Fej\'er convergent to
   $\Omega$ for any $q\in (0,\infty)$.
 \end{enumerate}
\end{remark}

From now on, $p$-Quasi-Fej\'er convergence will be called simply
Quasi-Fej\'er convergence, the exponent $p$ being $1$ unless
otherwise stated.

\subsection*{Point-to-set operators}

Let $X$, $Y$ be arbitrary sets. A point-to-set map $F:X\tos Y$ is a
function $F:X\to \wp(Y)$, where $\wp(Y)$ is the power set of $Y$, that
is, the family of all subsets of $Y$. If $F(x)$ is a singleton for all
$x$, that is, a set with just one element, one says that $F$ is
point-to-point.  Whenever necessary, we will identify a point-to-point
map $F:X\tos Y$ with the unique function $f:X\to Y$ such that
$F(x)=\{f(x)\}$ for all $x\in X$,

A point-to-set map $F:X\tos Y$ is \emph{$L$-Lipschitz} if
$X$ and $Y$ are normed vector spaces and,
\begin{equation}
  \label{eq:lip}
  \emptyset\neq F(x')\subset \{y+u\;|\; y\in F(x),\;u\in Y,\;
  \norm{u}\leq L\norm{x-x'}\},
  \qquad \forall x,x'\in X.
\end{equation}
Note that if $F$ is point-to-point and it is identified with a
function, then in the above definition we retrieve the classical
notion of a $L$-Lipschitz continuous function.

\subsection*{Maximal monotone operators and the $\varepsilon$-enlargement}

The $\varepsilon$-enlargement of a maximal monotone operators will
be used to define approximate resolvents in Section~\ref{sec:apres}.
In this section we review the definition of the $\varepsilon$-enlargement
and discuss those of its properties which will be used in the analysis
and applications of approximate resolvents.

From now on, $X$ is a real Hilbert space.
Recall that a point-to-set operator
$T:X\tos X$ is \emph{monotone} if
\[
\inner{x-y}{u-v}\geq 0\qquad \forall x,y\in X,\;u\in T(x),v\in T(y),
\]
and it is \emph{maximal monotone} if it is monotone and maximal
in the family of monotone operators in $X$ with respect to the partial
order of the inclusion.

Let $T:X\tos X$ be a maximal monotone operator.
Recall that the $\varepsilon$-enlargement~\cite{MR1463929} of $T$ is defined
\begin{equation}
  \label{eq:teps}
  T^{[\varepsilon]}(x)=\{v\;|\; \inner{x-y}{v-u}\geq -\varepsilon\},\qquad
  x\in X,\varepsilon\geq 0.
\end{equation} 
Now we state some elementary properties of the $\varepsilon$-enlargement
which follow trivially from the above definition and the basic properties
of maximal monotone operators. Their proofs
can be found in~\cite{MR1463929,MR1716028,MR1802238}.
\begin{proposition}
  \label{pr:teps.bas}
   Let $T:X\tos X$ be maximal monotone. Then
  \begin{enumerate}
  \item \label{it:tz.t}$T=T^{[0]}$;
  \item \label{it:te.t} if $0\leq\varepsilon_1\leq\varepsilon_2$, then
    $T^{[\varepsilon_1]}(x)\subset T^{[\varepsilon_2]}(x)$ for any
    $x\in X$;
  \item \label{it:lte} $\lambda\left(T^{[\varepsilon]}(x)\right)
    =(\lambda T)^{[\lambda\varepsilon]}(x)$ for any $x\in X$,
    $\varepsilon\geq0$ and $\lambda>0$;
  \item \label{it:demi} if $v_k\in T^{[\varepsilon_k]}(x_k)$ for
    $k=1,2,\dots$, $(x_k)$ converges weakly to $ x$, $(v_k)$ converges
    strongly to $v$ and $(\varepsilon_k)$ converges to $\varepsilon$,
    then $v\in T^{[\varepsilon]}(x)$;
  \item \label{it:te.sd} if $T=\partial f$, where $f$ is a proper
    closed convex function in $X$, then $\partial_\varepsilon
    f(x)\subset T^{[\varepsilon]}(x)=(\partial f)^{[\varepsilon]}(x)$
    for any $x\in X$, $\varepsilon\geq 0$.
   \end{enumerate}
\end{proposition}

The $\varepsilon$-enlargements of two operators can be ``added''  as
follows. This fact was proved in~\cite{MR1463929} in a finite dimensional
setting, but its extension to Hilbert and Banach spaces are straightforward.

\begin{proposition}
  \label{pr:teps.sum}
  It $T_1,T_2:X\tos X$ are maximal monotone and $T_1+T_2$ is also maximal
  monotone then, for any $\varepsilon_1,\varepsilon_2\geq 0$ and $x\in X$
  \[
  T_1^{[\varepsilon_1]}(x)+  T_2^{[\varepsilon_2]}(x)
  \subset   (T_1+T_2)^{[\varepsilon_1+\varepsilon_2]}(x).
  \]
\end{proposition}

Recall that a (maximal) monotone operator $A:X\to X$ is
\emph{$\alpha$-cocoercive} (for $\alpha>0$) if \begin{align*}
  \inner{x-y}{Ax-Ay}\geq \alpha\norm{Ax-Ay}^2,\qquad \forall x,y\in X.
\end{align*}
There is an interesting ``transportation formula'' for cocoercive operators.
This result was proved by R.D.C Monteiro and myself.

\begin{lemma}
  \label{lm:tr.cc}
  If ${A}:X\to X$ is $\alpha$-cocoercive, then for any $x,z\in X$,
  \[
  {A}(z)\in {A}^{[\varepsilon]}(x), \qquad\mbox{with } \varepsilon=
  \frac{\norm{x-z}^2}{4\alpha}.
\]
\end{lemma}

\begin{proof}
  Take $y\in X$. Then
  \begin{align*}
    \inner{x-y}{{A}z-{A}y}
    &=\inner{x-z}{{A}z-{A}y}+\inner{z-y}{{A}z-{A}y}\\
    &\geq \inner{x-z}{{A}z-{A}y}+\alpha\norm{{A}z-{A}y}^2\\
    &\geq -\norm{x-z}\norm{{A}z-{A}y}+\alpha\norm{{A}z-{A}y}^2,
\end{align*}
where the first inequality follows form the cocoercivity of ${A}$ and
the second one from Cauchy-Schwarz inequality. To end the proof, note that
\[
 -\norm{x-z}\norm{{A}z-{A}y}+\alpha\norm{{A}z-{A}y}^2\geq
\inf_{t\in\R}\alpha t^2-\norm{x-z}t
\]
and compute the value of the left hand-side of this inequality.
\end{proof}

 The usefulness of the $\sigma$-approximate resolvent (to be defined
in Section~\ref{sec:apres}) follows from
the  next elementary result, essentially proved in~\cite[Lemma 2.3,
Corollary 4.2]{MR1756912}.

\begin{lemma}
  \label{lm:app.res}
  Suppose that $T:X\tos X$ is maximal monotone, $x\in X$,
  $\lambda>0$ and $\sigma\geq 0$. If
  \begin{align*}
    \left\{
    \begin{array}{l}
    v\in T^{[\varepsilon]}(y),\\ \norm{\lambda v+y-x}^2+2\lambda\varepsilon
  \leq \sigma^2\norm{y-x}^2,
    \end{array}\right.
    \qquad\mbox{ and }\qquad z=x-\lambda v,&
  \end{align*}
  then 
  \(
  \norm{\lambda v}\leq (1+\sigma)\norm{y-x}\),
  \(\norm{z-y}\leq \sigma\norm{y-x}
  \)
   and for any $x^*\in T^{-1}(0)$
  \begin{align*}
  \norm{x^*-x}^2
  &\geq \norm{x^*-z}^2+\norm{y-x}^2-
  \bigg[\norm{\lambda v+y-x}^2+2\varepsilon\bigg]\\
  &\geq \norm{x^*-z}^2+(1-\sigma^2)\norm{y-x}^2.
  \end{align*}
\end{lemma}

\begin{proof}
  Since $\varepsilon\geq 0$ we have $\norm{\lambda v + y - x} \leq
  \sigma\norm{y-x}$. The two first inequalities of the lemma
  follows trivially from this inequality, triangle inequality and the
  definition of $z$.

  To prove the third inequality of the lemma, take $x^*\in T^{-1}(0)$.
  Direct combination of the algebraic identities
  \begin{align*}
    \norm{x^*-x}^2
    &=\norm{x^*-z}^2+2\Inner{x^*-y}{z-x}+2\Inner{y-z}{z-x}+\norm{z-x}^2
    \\&=\norm{x^*-z}^2+2\Inner{x^*-y}{z-x}+\norm{y-x}^2-\norm{y-z}^2
  \end{align*}
  with the definition of $z$ yields
  \[
  \norm{x^*-x}^2=\norm{x^*-z}^2+2\lambda
\Inner{x^*-y}{-v}+\norm{y-x}^2-\norm{\lambda v+y-x}^2.
  \]
  Using the inclusions $0\in T(x^*)$, $v\in T^{[\varepsilon]}(y)$ and the
  definition in~\eqref{eq:teps}, we conclude that 
  $ \inner{x^*-y}{0-v}\geq -\varepsilon$.
  To end the proof, of the third inequality,
  combine this inequality with the above equations.

  The last inequality follows trivially from the third one and the assumptions
  of the lemma.
\end{proof}

\section{A class of Fej\'er convergent methods}
\label{sec:class.fc}

Let $X$ be a Hilbert space and $\Omega\subset X$.
We are concerned with iterative methods for solving problem
\begin{equation}
  \label{eq:p1}
  x\in \Omega.
\end{equation}
These methods, in their exact or inexact form, generate sequences $(x_n)$ 
by means of the \emph{recursive inclusions}
\[
x_n\in F_n(x_{n-1})\mbox{ or }x_n\in F_n(x_{n-1})+r_n, \qquad n=1,2,\dots,
\]
respectively, where $F_1:X\tos X,F_2:X\tos X,\dots$ are point-to-set
maps, $r_1,r_2\dots$ are errors and $x_0\in X$ is a starting point.
The basic elements here are the set $\Omega$ and the family of
point-to-set maps $(F_n)$, which we will call the family of
\emph{iteration-maps}.

We will consider two properties of a general family of point-to-set
maps $(F_n:X\tos X)_{n\in\mathbb{N}}$ with respect to $\Omega\subset
X$:
\begin{description}
\item[P1:]  if $\hat x\in F_n(x)$ and $x^*\in\Omega$ then
     \[
     \norm{x^*-\hat x}\leq\norm{x^*-x};
    \]
\item[P2:] if $(z_k)_{k\in\mathbb{N}}$ converges weakly to $\bar z$,
  $\hat z_k\in F_{n_k}(z_k)$  for $n_1<n_2<\cdots$ and for some $x^*\in \Omega$ 
    \[
    \lim_{k\to\infty} \norm{x^*-z_k}-\norm{x^*-\hat z_k}=0,
    \]
    then $\bar z\in \Omega$.
\end{description}
Property {\bf P1} ensures that points in the image of $F_n(x)$ are
closer (or no more distant) to $\Omega$ than $x$.  Regarding
property {\bf P2}, note that (using property {\bf P1}) we have
\[
\norm{x^*-z_k}-\norm{x^*-\hat z_k}\geq 0.
\]
The left hand-side of the above inequality measures the progress of
$\hat z_k$ toward the solution $x^*$, as compared to $z_k$. Hence,
property {\bf P2} ensures that if the progress becomes ``negligible'',
then the weak limit point of $(z_k)$
belongs to $\Omega$.
\begin{theorem}
  \label{th:main}
  Suppose that $\Omega\subset X$ is non-empty and
  $\left( F_n:X\tos X\right)$ is a sequence of point-to-set maps
  which satisfies conditions {\bf P1}, {\bf P2} with respect to
  $\Omega$.
  
  If
   \[
    x_n\in F_n(x_{n-1})+r_n,\qquad \sum\norm{r_n}<\infty
  \]
  then $(x_n)$ is Quasi-Fej\'er convergent to $\Omega$, it converges
  weakly to some $\bar x\in\Omega$ and for any $w\in\Omega$ there
  exists $\lim_{n\to\infty} \norm{x^*-x_n}$.

  Moreover, if $r_n=0$ for
  all $n$, then $(x_n)$ is Fej\'er convergent to $\Omega$.
\end{theorem}

\begin{proof}
  To simplify the proof, define
  \[
  \hat x_n=x_n-r_n.
  \]  
  Take an arbitrary $x^*\in \Omega$. Since $\hat x_n\in F_n(x_{n-1})$, 
  $\norm{x^*-\hat x_n}\leq \norm{x^*-x_{n-1}}$,
  \[
  \norm{x^*-x_n}\leq \norm{x^*-\hat x_n}+\norm{r_n}
  \leq \norm{x^*-x_{n-1}}+\norm{r_n}
  \]
  and $(x_n)_{n\in\mathbb{N}}$ is Quasi-Fej\'er convergent to
  $\Omega$.  Therefore, by Proposition~\ref{pr:qf1}, this sequence is
  bounded and there exists $\lim_{n\to\infty} \norm{x^*-x_n}<\infty$.
  Using this fact, the above equation and the assumption of $(r_n)$
  being summable we conclude that
  \[
  \lim_{n\to\infty}\norm{x^*-x_{n-1}}-\norm{x^*-\hat x_n}=0.
  \]
  Since $(x_n)_{n\in\mathbb{N}}$ is bounded it has a
  weak cluster point, say $\bar x$ and there exists a subsequence
  $(x_{n_k})_{k\in\mathbb{N}}$ which converges weakly to $\bar x$.
  The above equation shows that, in particular
  \[
  \lim_{k\to\infty}\norm{x^*-x_{n_k-1}}-\norm{x^*-\hat x_{n_k}}=0
  \]
  Therefore, using {\bf P2}, the two above equations and the inclusion
  $\hat x_{n_k}\in F_{n_k}(x_{n_k-1})$, we conclude that $\bar x\in \Omega$.
  Hence, all weak cluster points of $(x_n)$ belong to $\Omega$.
  To end the proof, use Proposition~\ref{pr:qf2}
\end{proof}

Note that properties {\bf P1}, {\bf P2} are ``inherited'' by specializations.
We state formally this result and the proof, being quite trivial, is
be omitted

\begin{proposition}
  \label{pr:sp}
  If $\left( F_n:X\tos X\right)$ is a sequence of point-to-set maps
  which satisfies conditions {\bf P1}, {\bf P2} with respect to
  $\Omega\subset X$ and $\left( G_n:X\tos X\right)$ is a sequence
  of point-to-set maps such that, for any $x\in X$
  \[
  G_n(x)\subset F_n(x),\qquad n=1,2,\dots,
  \]
  then  $\left( G_n:X\tos X\right)$ also 
  satisfies conditions {\bf P1}, {\bf P2} with respect to
  $\Omega$.
\end{proposition}

What about compositions?  Suppose that $(F_n)$ is a sequence satisfying
{\bf P1}, {\bf P2} with respect to $\Omega\subset X$, and that
\[
F_n=G_n\circ H_n
\]
where $G_n:X\tos Y$, $H_n:Y\tos X$ and all $G_n$'s are $L$-Lipschitz
continuous ($Y$ is Hilbert). One may consider sequences
\begin{align}
  \label{eq:z}
  {y}_n\in H_n(x_{n-1})+{u}_n, \qquad
  x_n\in G_n({y}_n)+{r}_n,  
\end{align}
where $({u}_n)$ and $({r}_n)$ are summable. Since 
$x_n-{r}_n\in G_n({y}_n)$, using
also~\eqref{eq:lip}, we conclude that there exists $\hat x_n
\in G_n({y}_n-{u}_n)\subset G_n\circ H_n(x_{n-1})$
\[
\norm{\hat x_n-(x_n-{r}_n)}\leq L\;\norm{{u}_n}.
\]
Therefore, defining $s_n=x_n-\hat x_n$
 we conclude that
\[
x_n\in F_n(x_{n-1})+s_n,\qquad  
\sum\norm{s_n}\leq \sum L\norm{{u}_n}+\norm{{r}_n}<\infty.
\]
Therefore, if $\Omega\neq\emptyset$, a sequence $(x_n)$ generated as
in~\eqref{eq:z} converges weakly to some point $x^*\in\Omega$.

On may also consider compositions of $m+1$ maps
\[
F= G_{1,n}\circ G_{2,n}\dots\circ G_{m,n}\circ H_n
\]
adding summable errors in each stage, assuming each $G_{i,n}:Y_i\tos
Y_{i-1}$ to be $L$-Lipschitz continuous, $H_n:X\tos Y_m$, $Y_0=X$
etc.

\section{Approximate resolvents and the Hybrid Proximal-Extragradient
 Method}
\label{sec:apres}

In this section, first we define $\sigma$-approximate resolvents,
analyze some of their properties and study conditions under which
sequences of $\sigma$-approximate resolvents satisfy properties {\bf
  P1}, {\bf P2}. After that, we recall the definition of the Hybrid
Proximal-Extragradient method and show that $\sigma$-approximate
resolvents are the iteration maps of such method.
At the end of the section we discuss the incorporation of summable
errors to sequences of $\sigma$-approximate resolvents and to the
Hybrid Proximal-Extragradient method.

Recall that the \emph{resolvent} of a maximal monotone operator
$T:X\tos X$ is defined as
\begin{equation}
   \label{eq:def.res}
  J_T(x)=(I+T)^{-1}(x),\qquad x\in X.
\end{equation}
We shall consider approximations of the resolvent in the following
sense.

\begin{definition}
  \label{df:apres}
  The $\sigma$-\emph{approximate resolvent} of a maximal monotone
  operator $T:X\tos X$ is the point-to-set operator
  $J_{T,\sigma}:X\tos X$
  \begin{equation*}
    J_{T,\sigma}(x)=\left\{x-v\;\left|
    \begin{array}{l} \exists \varepsilon \geq 0, y\in X,\\
    v\in T^{[\varepsilon]}(y)\\
    \norm{v+y-x}^2+2\varepsilon\leq
    \sigma^2\norm{y-x}^2
    \end{array}
    \right\}\right.
  \end{equation*}
  where $\sigma\geq 0$.
\end{definition}

First, we analyze some elementary properties of approximate resolvents
and find a convenient expression for $J_{\lambda T,\sigma}$.  In
particular, we show that the $\sigma$-approximate resolvent is
indeed and extension (in the sense of point-to-set maps) of the
classical resolvent.

\begin{proposition}
  \label{pr:apr.bas}
  Let $T:X\tos X$ be maximal monotone. Then, for any $x\in X$,
  \begin{enumerate}
  \item 
    $J_{T,\sigma=0}(x)=\{J_T(x)\}$;
  \item
    \label{it:apr.mon}
    if $0\leq \sigma_1\leq \sigma_2$ then $J_{T,\sigma_1}(x) \subset
    J_{T,\sigma_2}(x)$;
  \item 
    \label{it:apr.lambda}
    for any $\lambda>0$ and $\sigma\geq 0$,
    \begin{align*}
      J_{\lambda T,\sigma}(x)&=\left\{x-\lambda v\;\left|
          \begin{array}{l} 
            \exists \varepsilon \geq 0, y\in X,\\
            v\in T^{[\varepsilon]}(y)\\
            \norm{\lambda v+y-x}^2+2\lambda \varepsilon\leq
            \sigma^2\norm{y-x}^2
          \end{array}
        \right\}\right.
    \end{align*}
  \end{enumerate}
\end{proposition}

\begin{proof}
  Items 1, 2 and 3 follow trivially from Definition~\ref{df:apres} and
  Proposition~\ref{pr:teps.bas}, items~\ref{it:tz.t}, \ref{it:te.t}
  and~\ref{it:lte}.
\end{proof}

Note that in view of item 1 of the above proposition, if point-to-set
operators which are point-to-point are identified with functions, we
have
\[
J_{T,0}=J_T\,.
\]
In view of item 3,
\[
J_{\lambda T,\sigma}=\left\{z\in X;\left|
          \begin{array}{l} 
            \exists \varepsilon \geq 0, y\in X,\\
            \displaystyle \frac{x-z}{\lambda}\in  T^{[\varepsilon]}(y)\\
            \norm{y-z}^2+2\lambda \varepsilon\leq
            \sigma^2\norm{y-x}^2
          \end{array}
        \right\}\right.
\]

The next theorem is the main result of this section and states that, in
some sense, approximate resolvents are ``almost as good'' as
resolvents for finding zeros of maximal monotone operators, that is,
for solving problem~\eqref{eq:p1} with $\Omega=\{x\:|\; 0\in T(x)\}=T^{-1}(0)$.

\begin{theorem}
  \label{th:hpe}
  Suppose that $T:X\tos X$ is maximal monotone, $\sigma\in [0,1)$,
  $\underline\lambda>0$ and $(\lambda_k)$ is a sequence in
  $[\underline\lambda, \infty)$. Then, the sequence of point-to-set
  maps
  \[
  \left(J_{\lambda_kT,\sigma}\right)_{k\in\mathbb{N}}
  \]
  satisfies properties \textbf{P1}, \textbf{P2} with respect to
  $\Omega=\{x\in X\;|\; 0\in T(x)\}=T^{-1}(0)$.
\end{theorem}

\begin{proof}
  Suppose that $\hat x\in J_{\lambda_kT,\sigma}(x)$. This means
  that there exists $y,v\in X$, $\varepsilon\geq0$ such that
  \[
  \hat x=x-\lambda_kv,\quad v\in T^{[\varepsilon]}(x),\quad
  \norm{\lambda_kv+y-x}^2+2\lambda\varepsilon\leq \sigma^2\norm{y-x}^2.
  \]
  Therefore, using Lemma~\ref{lm:app.res} we conclude that for any 
  $x^*\in =T^{-1}(0)$,
  \[
  \norm{x^*-x}^2\geq \norm{x^*-\hat x}^2+(1-\sigma^2)\norm{y-x}^2
   \geq \norm{x^*-\hat x}^2
  \]
  which proves that the family $(J_{\lambda_kT,\sigma})$ satisfies \textbf{P1}.

  Now we prove {\bf P2}. Suppose that $(z_k)$ converges weakly to
  $\bar z$, $\hat z_k\in J_{\lambda_{n_k}T,\sigma}(z_k)$, $0\in
  T(x^*)$ and
  \begin{equation}
    \label{eq:toz}
  \lim_{k\to\infty}\norm{x^*-z_k}-\norm{x^*-\hat z_k}=0.
  \end{equation}
  To simplify the proof, let $\mu_k=\lambda_{n_k}\geq \underline\lambda$.
  For each $k$ there exists $v_k,y_k\in X$, $\varepsilon_k\geq 0$ such that
  \begin{equation}
    \label{eq:aux.2}
  \hat z_k=z_k-\mu_kv,\quad v_k\in T^{\varepsilon_k}(z_k),\quad
  \norm{\mu_kv_k+y_k-z_k}^2+2\mu_k\varepsilon\leq \sigma^2\norm{y_k-z_k}^2.
  \end{equation}
  Using again Lemma~\ref{lm:app.res}, we conclude that
  \[
  \norm{x^*-z_k}^2\geq \norm{x^*-\hat z_k}^2+(1-\sigma^2)\norm{y_k-z_k}^2.
  \]
  Therefore
  \begin{align*}
    (1-\sigma^2)\norm{y_k-z_k}^2&\leq \norm{x^*-z_k}^2-\norm{x^*-\hat z_k}^2\\
    &=(\norm{x^*-z_k}-\norm{x^*-\hat z_k})
    (\norm{x^*-z_k}+ \norm{x^*-\hat z_k}).
  \end{align*}
  Since $(z^k)$ is weakly convergent, it is also bounded. Taking this
  fact into account and using the above equation and \eqref{eq:toz}
  we conclude that 
  \[
  \lim_{k\to\infty}\norm{y_k-z_k}=0.
  \]
  So, $(y_k)$ also converges weakly to $\bar z$. Since $\varepsilon_k\geq 0$,
  using the last relation in~\eqref{eq:aux.2} we conclude that
  \[
  \mu_k\varepsilon_k\leq \frac{\sigma^2}{2}\norm{y_k-z_k}^2, \qquad
  \norm{\mu_kv_k}\leq (1+\sigma)\norm{y_k-z_k}.
\]
  Therefore, since $(\mu_k)$ is bounded away from $0$,
  \[ \lim_{k\to\infty}\varepsilon_k=0, \quad 
    \lim_{k\to\infty}v_k=0
\]
and $0\in T^{[0]}(\bar z)=T(\bar z)$.
\end{proof}

The Hybrid Proximal-Extragradient/Projection methods were introduced
in \cite{MR1713951,MR1756912,MR1871872}.  These methods are variants
of the Proximal Point method which use relative error tolerances for
accepting inexact solutions of the proximal sub-problems.  Here we are
concerned with the variant introduced in~\cite{MR1756912}, which will
be called, from now on, the Hybrid Proximal-Extragradient (HPE)
method.  It solves iteratively the problem
\begin{equation}
  \label{eq:ip}
  0\in T(x),  
\end{equation}
where $T:X\tos X$ is maximal monotone. This method proceeds as follows.
\\
\\{\sc Algorithm: (projection free) HPE method}~\cite{MR1756912}:\\
Choose $x_0\in X$, $\sigma\in[0,1)$, $\underline\lambda>0$ and for
$k=1,2,\dots$\\
a) Choose $\lambda_k\geq\underline\lambda$ and find/compute
$v_k,y_k\in X$, $\varepsilon\geq 0$ such that
\begin{align*}
  v_k\in T^{\varepsilon_k}( y_k),\qquad 
  \norm{\lambda_k v_k+y_k-x_{k-1}}^2+2\lambda_k\varepsilon_k\leq
\sigma^2\norm{y_k-x_{k-1}}^2 
\end{align*}
b) Set $x_k=x_{k-1}-\lambda_kv_k$
\\
\\
To generate iteratively sequences by means of approximate resolvents
is equivalent to apply the HPE method in the following sense.

\begin{proposition}
  \label{pr:hy.appj}
  Let $T:X\tos X$ be maximal monotone, $\sigma\geq 0$,
  $\underline{\lambda}>0$ and $(\lambda_k)$ be sequence in
  $[\underline{\lambda},\infty)$.

  A sequence $(x_k)$ satisfies the recurrent inclusion
  \begin{equation*}
    x_k\in J_{\lambda_k T,\sigma}(x_{k-1}),\qquad k=1,2,\dots    
  \end{equation*}
  if and only if there exists sequences $(y_k)$, $(v_k)$,
  $(\varepsilon_k)$ which, together with the sequences $(x_k)$,
  $(\lambda_k)$ satisfy steps a) and b) of the HPE method.
\end{proposition}

\begin{proof}
  Use Definition~\ref{df:apres} and Proposition~\ref{pr:apr.bas}
  item~\ref{it:apr.lambda}.
\end{proof}

Convergence of the HPE method perturbed by a summable sequence of
errors was proved directly in~\cite{MR1850681}.
Here we see that it can be easily and effortlessly derived as a
particular case of a generic convergence result,
combining Proposition~\ref{pr:hy.appj}
with Theorem~\ref{th:hpe}.

\begin{corollary}
  \label{cr:hy.se}
  If $T:X\tos X$ is maximal monotone, $T^{-1}(0)\neq\emptyset$,
  $\underline\lambda>0$, $\sigma\in[0,1)$, for $k=1,2,\dots$
  \begin{align*}
&\lambda_k\geq \underline\lambda\\
& v_k\in T^{[\varepsilon_k]}(\tilde x_k),\;
\norm{\lambda_kv_k+\tilde x_k-x_{k-1}}^2+2\lambda_k\varepsilon_k\leq
\sigma\norm{\tilde x_k-x_{k-1}}^2\\
&x_k=x_{k-1}-\lambda_kv_k+r_k
  \end{align*}
and $\sum\norm{r_k}<\infty$, then $(x_k)$ (and $(\tilde x_k)$) converges
weakly to a point $\bar x\in T^{-1}(0)$.
\end{corollary}

\begin{corollary}
  \label{cr:hy.se2}
  If $T:X\tos X$ is maximal monotone, $T^{-1}(0)\neq\emptyset$,
  $\overline{\lambda}\geq \underline\lambda>0$, $\sigma\in[0,1)$, for
  $k=1,2,\dots$
  \begin{align*}
&\overline{\lambda}\geq \lambda_k\geq \underline\lambda\\
& v_k\in T^{[\varepsilon_k]}(\tilde x_k),\;
\norm{\lambda_kv_k+\tilde x_k-x_{k-1}}^2+2\lambda_k\varepsilon_k\leq
\sigma\norm{\tilde x_k-x_{k-1}}^2\\
&x_k=x_{k-1}-\lambda_k(v_k+r_k)
  \end{align*}
and $\sum\norm{r_k}<\infty$, then $(x_k)$ (and $(\tilde x_k)$) converges
weakly to a point $\bar x\in T^{-1}(0)$.
\end{corollary}

\section{The Forward-Backward splitting method}
\label{sec:fv-bw}

We will prove in this section that the iteration maps of the
Forward-Backward splitting method are specializations or selections of
$\sigma$-approximate resolvents and the sequence of
iteration maps satisfies properties {\bf P1}, {\bf P2}.
Equivalently, the Forward-Backward splitting method is a particular
instance  of the HPE method.
Observe that, as a consequence, sequences generated by the inexact
Forward-Backward splitting methods with summable errors still converge
weakly to solutions of the inclusion problem, if any.
This convergence result was previously obtained in \cite{MR2115266} by
a detailed analysis of the Forward-Backward splitting method. Here we
see that it can be easily and effortlessly derived as a particular
case of a generic convergence result.

The Forward-Backward Splitting method solves the inclusion problem
\[
0\in (A+B)x
\]
where
\\
\\
{\bf f1)} ${A}:X\to X$ is $\alpha$-cocoercive, $\alpha>0$;
\\
{\bf f2)} ${B}:X\tos X$ is maximal monotone.
\\
\\
This method proceeds as follows:
\\
\\
{\sc Forward-Backward Splitting method}\\
0) Initialization: Choose $0<\underline\lambda\leq\bar\lambda< 2 \alpha$
and $x_0\in X$;
\\
1) for $k=1,2,\dots$\\
a) choose $\lambda_k\in [\underline\lambda,\bar\lambda]$ and define
\begin{align}
  \nonumber
  x_k&=(I+\lambda_k{B})^{-1}(x_k-\lambda_k{A}(x_{k-1}))
  \\&
  \label{eq:fb.im}
  =J_{\lambda_k{B}}\circ (I-\lambda_k{A})\;\,( x_{k-1}).
\end{align}

Note that the generic iteration map of the Forward-Backward method is
\begin{equation}
  \label{eq:fb.gim}
  J_{\lambda {B}}\circ (I-\lambda {A})
\end{equation}
with $\lambda=\lambda_k$ in the $k$-th iteration.

\begin{lemma}
  \label{lm:fb}
  If $A,B$ satisfy {\bf f1} and  {\bf f2}
  then, for any $\lambda>0$ and $x\in X$,
  \[
  J_{\lambda {B}}\circ(I-\lambda {A})(x)\in J_{\lambda(A+B),\sigma}
  (x),
  \]
  with   $\sigma=\sqrt{\lambda/(2\alpha)}$.
\end{lemma}

\begin{proof}
  Take $x\in X$ and let $z=J_{\lambda {B}}\circ(I-\lambda {A})(x)$. This
  means that
  \[
  b:=\lambda^{-1}(x-\lambda {A}(x)-z)\in {B}(z).
  \]
  Define $\varepsilon=\norm{x-z}^2/(4\alpha)$, $v=A(x)+b$.  Using
  Lemma~\ref{lm:tr.cc} we conclude that ${A}(x)\in
  {A}^{[\varepsilon]}(z)$. Therefore, combining this result with these
  two definitions, the above equation, Proposition~\ref{pr:teps.sum}
  and Proposition~\ref{pr:teps.bas} item~\ref{it:tz.t}, we conclude
  that
  \begin{align*}
    & v\in (A^{[\varepsilon]}+b)(z)\subset (A+B)^{[\varepsilon]}(z),
    \quad\norm{\lambda v+z-x}^2+2\lambda\varepsilon=\sigma^2\norm{z-x}^2\\
    & z=x-\lambda v,
  \end{align*}
 which, together with Proposition~\ref{pr:apr.bas} item 3, proves the lemma.  
\end{proof}

\begin{corollary}
  \label{cr:fb}
  Let $A,B$ be as in {\bf f1}, {\bf f2} and $\underline\lambda$,
  $\bar\lambda$ and $(\lambda_k)$, $(x_k)$ be as in the Forward
  Backward method.
  Define
  \[
  \sigma=\sqrt{\bar \lambda/(2\alpha)}.
  \]
  Then $0<\sigma<1$ and for any $x\in X$ 
  \begin{equation}
    \label{eq:fb.im.1}
    J_{\lambda_k{B}}\circ (I-\lambda_k{A})(x)\in J_{\lambda_k(A+B),\sigma\,}(x)
  ,\qquad k=1,2,\dots
  \end{equation}
  In particular
  \begin{equation}
    \label{eq:fb.im.2}
    x_k= J_{\lambda_k{B}}\circ (I-\lambda_k{A})(x_k)\in
    J_{\lambda_k(A+B),\sigma\,}(x_k)
    ,\qquad k=1,2,\dots
  \end{equation}
  and the sequence of maps $(J_{\lambda_k{B}}\circ(I-\lambda_k{A}))$
  satisfies properties {\bf P1}, {\bf P2} with respect to
  $(A+B)^{-1}(0)$.
\end{corollary}

\begin{proof}
  The bounds for $\sigma$ follow trivially from its definition and
  the choices for $\underline{\lambda}$, $\bar\lambda$ in the Forward-Backward
  method.

  Define $ \sigma_k=\sqrt{\lambda_k/(2\alpha)}$ for $k=1,2,\dots$
  Since $\lambda_k\in[\underline{\lambda},\bar\lambda]$,
  $0<\sigma_k\leq\sigma$ for all $k$.  Therefore, using also
  Lemma~\ref{lm:fb} and Proposition~\ref{pr:apr.bas}
  item~\ref{it:apr.mon}, we conclude that for any $x\in X$
  \[
  J_{\lambda_k{B}}\circ (I-\lambda_k{A})(x)\in
  J_{\lambda_k(A+B),\sigma_k}(x)\subset J_{\lambda_k(A+B),\sigma}(x),
  \quad k=1,2,\dots
  \]
  The equality in \eqref{eq:fb.im.2} follows trivially from the
  definition of the Forward-Backward method, while the inclusion
  follows from the above equation. To end the proof, note that
  $0<\underline\lambda<\lambda_k$ for all $k$, and use
  Theorem~\ref{th:hpe}, Proposition~\ref{pr:sp} and the above
  equation.
\end{proof}

\begin{proposition}
  \label{pr:fb}
Let  $(\lambda_k)$, $(x_k)$ be sequences generated by the
Forward-Backward Splitting method. Define
\begin{align*}
\sigma=\sqrt{\frac{\overline\lambda}{2\alpha}},\;\;
v_k=\lambda_k^{-1}(x_{k-1}-x_k),\;\;
\varepsilon_k=\frac{\norm{x_k-x_{k-1}}^2}{4\alpha},
\; k=1,2,\dots
\end{align*}
Then $0<\sigma<1$ and for $k=1,2,\dots$
\begin{align*}
  & v^k\in ({B}+{A})^{[\varepsilon_k]}(x_k), \quad 
  \norm{\lambda_kv_k+x_k-x_{k-1}}^2+2\lambda_k\varepsilon_k
  \leq \sigma\norm{x_k-x_{k-1}}^2\\
  &x_k=x_{k-1}-\lambda_kv_k.
\end{align*}
In particular, the Forward-Backward splitting method above defined is
a particular case of the HPE method with $\sigma\in(0,1)$.
\end{proposition}

\begin{proof}
  See the proofs of Lemma~\ref{lm:fb} and Corollary~\ref{cr:fb}.
\end{proof}

\section{Tseng's Modified Forward-Backward splitting method}
\label{sec:ts.mfb}

In~\cite{MR1756912} it was proved that
Tseng's Modified Forward-Backward splitting method~\cite{MR1741147}
is a particular case of the HPE method.
Here we will cast this result in the framework of approximate
resolvents,
prove that the iteration maps of the Tseng's Modified
Forward-Backward splitting method are specializations or selections of
$\sigma$-approximate resolvents and that the sequence of its iteration maps
satisfies properties {\bf P1}, {\bf P2}.
Observe that, as a consequence, sequences generated by
inexact Tseng's Modified Forward-Backward splitting method
with summable errors still converge
weakly to solutions of the inclusion problem, if any.
This result follows also from the fact that Tseng's method is a
particular case of the HPE (as proved in~\cite{MR1756912}) and that
the HPE with summable errors converges (as proved
in~\cite{MR1850681}).
Convergence of Tseng's  Modified Forward-Backward splitting method
with summable errors  was 
obtained in \cite{MR2854581} by a detailed
analysis of the Tseng's Modified Forward-Backward splitting method.
Here we  see
that this result can be easily and effortlessly derived as a particular case of a generic
convergence result.

In this section we  consider the inclusion problem
\[
0\in (A+B)\,x
\]
where 
\\
\\
{\bf t1)} ${A}:X\to X$ is monotone
 and $L$-Lipschitz continuous ($L>0$);\\
{\bf t2)} ${B}:X\tos X$ is maximal monotone.
\\
\\
The exact Tseng's Modified Forward-Backward Splitting method
(without the auxiliary projection step)
proceeds as follows:
\\
\\
\noindent
{\sc Tseng's Modified Forward-Backward method}~\cite{MR1741147}\\
Choose $0<\underline\lambda\leq\bar\lambda< 1/L$ and $x_0\in X$;\\
for $k=1,2,\dots$\\
a) choose $\lambda_k\in[\underline{\lambda},\bar\lambda]$ and
compute
\[
 y_k=( I+\lambda_k {B})^{-1}(x_{k-1}-\lambda_k {A}(x_{k-1})),\qquad
x_k=y_{k-1}-\lambda_k({A}(y_k)-{A}(x_{k-1})).
\]
\\
\\
In order to cast this method in the formalism of
Section~\ref{sec:class.fc}, define for $\lambda>0$
\begin{align}
\label{eq:ts.h}
&
H_\lambda:X\to X\times X,
&&
H_\lambda(x)=(x,J_{\lambda {B}}(x-\lambda {A}(x)))
\\
\label{eq:ts.g}
&
G_\lambda:X\times X\to X,
&&
G_\lambda(x,y)=y-\lambda ({A}(y)-{A}(x)))
\end{align}
Note that the second component of the (generic) operator $H_{\lambda}$
is $J_{\lambda {B}}\circ (I-\lambda {A})$ which is the generic
iteration map of the Forward-Backward method in \eqref{eq:fb.gim}.
Trivially,
\begin{equation}
  \label{eq:ts:xk}
  x_k=G_{\lambda_k}\circ H_{\lambda_k}\;(x_{k-1}).
\end{equation}
The next two result were essentially proved in \cite{MR1756912},
in the context of the Hybrid Proximal-Extragradient Method.
We will state and prove it in the context of $\sigma$-approximate
resolvents.

\begin{lemma}
  \label{lm:ts}
  If $A,B$  satisfy assumptions {\bf t1}, {\bf t2},
  then, for any $\lambda>0$ and $x\in X$
  \[
  G_\lambda\circ H_\lambda(x)\in J_{A+B,\sigma\;}(x)
   \]
  for $\sigma=\lambda L$.
\end{lemma}

\begin{proof}
  Take $x\in X$ and let
  \[  
  y=J_{\lambda {B}}(x-\lambda {A}(x)),\qquad
  z=y+\lambda({A}(y)-{A}(x)).
  \]
  Note that $z=G_\lambda\circ H_\lambda(x)$. Using the definition of $y$
  we have
  \[
  a:=\lambda^{-1}(x-\lambda {A}(x)-y)\in {B}(y).
  \] 
  Therefore,
  \begin{align*}
    v:=a+{A}(y)\in (A+B)(y),\qquad 
    \norm{\lambda v+x-y}^2&=\norm{\lambda({A}(y)-{A}(x))}^2\\
    &\leq (\lambda L)^2\norm{y-x}^2,
  \end{align*}
  where the inequality follows from assumption {\bf t2)}.
  To end the proof, note that $z=x-\lambda v$.
\end{proof}

\begin{corollary}
  \label{cr:ts}
  Let $A,B$ be as in {\bf t1}, {\bf t2} and $0<\underline\lambda<
  \bar\lambda<2\alpha$ and $(\lambda_k)$, $(x_k)$ be as in 
  Tseng's Modified Forward-Backward method.
  Define
  \[
  \sigma=\bar \lambda L.
  \]
  Then $0<\sigma<1$ and  for any $x\in X$ 
  \begin{equation}
    \label{eq:ts.im.0}
 G_{\lambda_k}\circ H_{\lambda_k}(x)\in J_{\lambda_k(A+B),\sigma\,}(x)
  ,\qquad k=1,2,\dots
  \end{equation}
  In particular
  \begin{equation}
    \label{eq:ts.im}
    x_k= G_{\lambda_k}\circ H_{\lambda_k} (x_{k-1})\in
    J_{\lambda_k(A+B),\sigma\,}(x_{k-1})
    ,\qquad k=1,2,\dots
  \end{equation}
  and the sequence of maps $( G_{\lambda_k}\circ H_{\lambda_k})$ satisfies
  properties {\bf P1}, {\bf P2} with respect to $(A+B)^{-1}(0)$.
\end{corollary}

\begin{proof}
  The bounds for $\sigma$ follow trivially from its definition and
  the choices for $\underline{\lambda}$ and $\bar\lambda$ in 
  Tseng's Forward-Backward
  method.

  Define $ \sigma_k=\lambda_kL$ for $k=1,2,\dots$  Since
  $\lambda_k\in[\underline{\lambda},\bar\lambda]$ we have
  $0<\sigma_k\leq\sigma$ for all $k$.  Therefore, using also
  Lemma~\ref{lm:ts} and Proposition~\ref{pr:apr.bas}
  item~\ref{it:apr.mon}, we conclude that for any $x\in X$
  \[
  J_{\lambda_k{B}}\circ (I-\lambda_k{A})(x)\in
  J_{\lambda_k(A+B),\sigma_k}(x)\subset J_{\lambda_k(A+B),\sigma}(x),
  \quad k=1,2,\dots
  \]
  The equality in \eqref{eq:ts.im} follows trivially from the
  definition of the Tseng's Modified Forward-Backward method, while the
  inclusion follows from the above equation. To end the proof, note
  that $0<\underline\lambda<\lambda_k$ for all $k$, and use
  Theorem~\ref{th:hpe}, Proposition~\ref{pr:sp} and the above
  equation.
\end{proof}

Note that for $0<\lambda\leq\bar\lambda$, the maps $H_\lambda$, $G_\lambda$
are Lipschitz continuous with constant
\[
2+\bar\lambda L,\qquad 1+2\bar\lambda L,
\]
respectively. Hence, this method can be perturbed by  summable sequences
of errors in the evaluations of the resolvents $J_{\lambda_k {B}}$ and/or
in the evaluation of ${A}(x_k)$, ${A}(y_k)$ etc, and will still converge weakly
to a solution, if any exists.

\section{Korpelevich's method}
\label{sec:kp}

In~\cite{MR2721154} it was proved that Korpelevich's method, with
fixed stepsize, is a particular case of the HPE method.  The extension
of this result for variable stepsizes is trivial, and here we will
analyze such an extension in the framework of approximate resolvents.
Observe that, as a consequence, sequences generated by
inexact Korpelevich's method
with summable errors still converges
weakly to solutions of the inclusion problem, if any.

In this section we  consider the inclusion problem
\[
0\in A(x)+N_C(x)
\]
where 
\\
\\
{\bf k1)} ${A}:X\to X$ is monotone and $L$-Lipschitz continuous ($L>0$);
\\
{\bf k2)} $N_C$ is the normal cone operator of $C\subset X$,
a non-empty closed convex set.
\\
\\
\noindent
{\sc Korpelevich's method}\\
Choose $0<\underline\lambda\leq\bar\lambda< 1/L$ and $x_0\in X$;\\
for $k=1,2,\dots$\\
a) choose $\lambda_k\in [\underline\lambda,\bar\lambda]$ and define
\begin{equation}
  \label{eq:kp}
  y_k=P_C(x_{k-1}-\lambda_kF(x_{k-1})),\qquad
  x_k=P_C(x_{k-1}-\lambda_kF(y_k)),
\end{equation}
where $P_C$ stands for the orthogonal projection onto $C$.
\\
\\
In order to cast this method in the formalism of Section~\ref{sec:class.fc},
define for $\lambda>0$
\begin{align}
\label{eq:kp.h}
&
H_\lambda:X\to X\times X,
&&
H_\lambda(x)=(x,P_C(x-\lambda {A}(x)),
\\
\label{eq:kp.g}
&
G_\lambda:X\times X\to X,
&&
G_\lambda(x,y)=P_C(x-\lambda {A}(y)).
\end{align}
Observe that since $P_C=J_{\lambda N_C}$, the 
second component of the (generic) operator $H_{\lambda}$ is
$J_{\lambda {B}}\circ (I-\lambda {A})$ with ${B}=N_C$ which is
the generic
iteration map of the forward backward method in \eqref{eq:fb.gim}
(with ${B}=N_C$). Note also that the
map $H_\lambda$ above defined
has an equivalent expression
\[
H_\lambda(x)=(x,J_{\lambda N_C}(x-\lambda {A}(x)))
\]
which can be obtained by setting ${B}=N_C$ in \eqref{eq:ts.h} ${B}=N_C$.
Trivially,
\begin{equation}
  \label{eq:kp.im}
  x_k=G_{\lambda_k}\circ H_{\lambda_k}(x_{k-1}),\qquad k=1,2,\dots
\end{equation}
The next two result were essentially proved in \cite{MR2721154},
in the context of the Hybrid Proximal-Extragradient Method.
We will state and prove them in the context of $\sigma$-approximate
resolvents.
\begin{lemma}
  \label{lm:kp}
  If ${A}$ and $C$ satisfy assumptions {\bf k1} and {\bf k2},
  then, for any $\lambda>0$ and $x\in X$
  \[
  G_\lambda\circ H_\lambda(x)\in J_{{A}+N_C,\sigma\;}(x)
   \]
  for $\sigma=\lambda L$.
\end{lemma}

\begin{proof}
  Take $x\in X$ and let
  \begin{align*}
    y&=P_C(x-\lambda {A}(x)),\quad z=P_C(x-\lambda {A}(y)).
  \end{align*}
  Note that $z= G_\lambda\circ H_\lambda(x)$. 
  Define
  \begin{align*}
    \eta&=\frac 1\lambda(x-\lambda {A}(x)-y),\\
    \nu&=\frac 1\lambda(x-\lambda {A}(y)-z),\quad
    \varepsilon=\inner{\nu}{z-y},\quad
    v=\nu+{A}(y).
  \end{align*}
  Trivially, $\eta\in N_C(y)$ and $\nu\in N_C(z)=\partial
  \delta_C (z)$. 
  Therefore,
  \[
  \nu\in\partial_\varepsilon\delta_C(y)\subset
  (\partial\delta_C)^{[\varepsilon]}(y)= (N_C)^{[\varepsilon]}(y)
  \]
  and
  \begin{equation}
    \label{eq:kk1}
    v\in ({A}+N_C)^{[\varepsilon]}(y),\qquad z=x-\lambda v.
  \end{equation}
  Therefore
  \begin{align*}
    \norm{\lambda v+y-x}^2+2\lambda\varepsilon&=\norm{y-z}^2+2
   \lambda\inner{\nu}{z-y}\\
  &=\norm{y-z}^2+
    2\lambda\inner{\nu-\eta}{z-y}+
    2\lambda\inner{\eta}{z-y}
\\
  &\leq \norm{y-z}^2+
    2\lambda\inner{\nu-\eta}{z-y},
  \end{align*}
  where the inequality follows from the inclusions  $\eta\in N_C(y)$,
  $z\in C$.
  Direct algebraic manipulations yield
\begin{align*}
   \norm{y-z}^2+
    2\lambda\inner{\nu-\eta}{z-y}&=
  \norm{\lambda(\nu-\eta)+z-y}^2-\norm{\lambda(\nu-\eta)}^2\\
   &\leq  \norm{\lambda(\nu-\eta)+z-y}^2\\
   &=\norm{\lambda({A}(x)-{A}(y))}^2.
  \end{align*}
  Combining the two above equations, and using assumption {\bf k1},
  we conclude that
  \[
   \norm{\lambda v+y-x}^2+2\lambda\varepsilon\leq (\lambda L)^2\norm{y-x}^2.
  \]
  The conclusion follows combining this inequality with \eqref{eq:kk1}.
\end{proof}

\begin{corollary}
  \label{cr:kp}
  Let ${A}, C$ be as in {\bf k1}, {\bf k2}, and $0<\underline\lambda<
  \bar\lambda<2\alpha$ and $(\lambda_k)$, $(x_k)$ be as in Korpelevich's
  method.
  Define
  \[
  \sigma=\bar \lambda L.
  \]
  Then $0<\sigma<1$ and for any $x\in X$ 
  \begin{equation*}
    G_{\lambda_k}\circ H_{\lambda_k}(x)\in J_{\lambda_k(A+B),\sigma\,}(x)
    ,\qquad k=1,2,\dots
  \end{equation*}
  In particular
  \begin{equation*}
    x_k= G_{\lambda_k}\circ H_{\lambda_k} (x_{k-1})\in
    J_{\lambda_k(A+B),\sigma\,}(x_{k-1})
    ,\qquad k=1,2,\dots
  \end{equation*}
  and the sequence of maps $( G_{\lambda_k}\circ H_{\lambda_k})$ satisfies
  properties {\bf P1}, {\bf P2} with respect to $({A}+N_C)^{-1}(0)$.
\end{corollary}

\begin{proof}
  Use Lemma~\ref{lm:kp} and the same reasoning as in
  corollaries~\ref{cr:fb} and \ref{cr:ts}.
\end{proof}

Endowing $X\times X$ with the canonical inner product 
of Hilbert space products
\[
\inner{(x,y)}{(x',y')}=\inner{x}{x'}+\inner{y}{y'},
\]
it is trivial to check that for $0<\lambda\leq\bar\lambda$,
the maps $H_\lambda$ and $G_\lambda$ are Lipschitz continuous with
constants
\[
2+\bar\lambda L,\qquad 1+\bar\lambda L,
\]
respectively. Hence, one can analyze 
Korpelevich's method with (summable) errors in the projections and/or
evaluations of ${A}$ etc.

\section{Discussion}
\label{sec:conc}

We provided a general definition of generic methods by means
of recursive inclusions and sequences of point-to-set maps.
Using this formulation, we 
defined two properties of those maps which guarantee that
the associated method is Fej\'er convergent and
generates sequences 
which converge to a solution, if any, even when perturbed
by summable errors.

We think these results obviate the summable error convergence analysis
of a number of convergent Fej\'er methods.

The framework for the analysis of Fej\'er convergent methods introduced
here is, of course, not general enough to encompasses \emph{all}
of these methods. Indeed, if $X=\R$, $\Omega=\{0\}$ and
\[F(x)=
\begin{cases}
  -x& x>0,\\
   x/2,& x\leq 0
\end{cases}
\]
then any sequence $(x_n)$ satisfying $x_n=F(x_{n-1})$ is Fej\'er
convergent to $\{0\}$ and converges to $0$. However, the sequence
$(F_n=F)$ does not satisfies {\bf P2}.

It has been since long recognized that Korpelevich's method (and may be
even the Forward-Backward method) was an ``inexact'' version of the
proximal point method. However, the nature and degree of this
``inexactness'' were not known.
We provided a formal definition of approximate solutions of the prox by
means of the $\sigma$-approximate resolvent which, while encompassing
many classical decomposition schemes, also guarantees weak convergence
of sequences generated by such approximate resolvents (even in the
presence of additional summable errors).

\def\cprime{$'$}

\end{document}